\documentclass[preprint,12pt]{elsarticle}

\usepackage{placeins}
\usepackage{eurosym}
\usepackage{url}

\usepackage{fancyhdr}
\usepackage[utf8x]{inputenc} 
\usepackage{indentfirst} 
\usepackage{graphicx} 
\usepackage{newlfont}
\usepackage{booktabs}
\usepackage{blindtext} 
\usepackage{multicol}
\usepackage{xcolor}
\usepackage{verbatim}
\usepackage{cite}
\usepackage{amsmath}
\usepackage{amssymb}
\usepackage{amsthm}
\usepackage{amsfonts}
\usepackage{thmtools}	
\theoremstyle{definition}
\newtheorem{theorem}{Theorem}

\newtheorem{proposition}[theorem]{Proposition}

\newtheorem{definition}[theorem]{Definition}
\newtheorem{example}[theorem]{Example}
\newtheorem{remark}[theorem]{Remark}

\usepackage{amsrefs}
\usepackage{enumerate}
\usepackage{cite}
\usepackage{xcolor}
\usepackage[final]{listings}
\usepackage{hyperref}
\usepackage{cleveref}

\usepackage{latexsym}
\usepackage{listings}
\usepackage{chngcntr}
\usepackage{subfiles}
\usepackage{subcaption}

\usepackage{pdfpages}
\usepackage{tensor}
\usepackage{soul}
\usepackage{pgf, tikz}
\usetikzlibrary{arrows, automata}
\usetikzlibrary{topaths,calc}

\usetikzlibrary{fit}
\tikzset{%
  highlight/.style={rectangle,rounded corners,fill=red!15,draw,
    fill opacity=0.5,thick,inner sep=0pt}
}

\newcommand{\C}{\mathbb{C}}

\newcommand{\K}{\mathbb{K}}

\newcommand{\Nc}{\mathcal{N}}

\journal{Advances in Applied Mathematics}

\begin{document}

\begin{frontmatter}



\title{Product of Tensors and Description of Networks}

\author[label1]{Luca Chiantini}
\ead{luca.chiantini@unisi.it}
\author[label2]{Giuseppe Alessio D'Inverno}
\ead{gdinvern@sissa.it}
\author[label1]{Sara Marziali}
\ead{sara.marziali@student.unisi.it}
\affiliation[label1]{organization={Department of Information Engineering and Mathematics, University of Siena},
            addressline={Via Roma 56},
            city={Siena},
            postcode={58100},
            country={Italy}}

\affiliation[label2]{organization={International School for Advanced Studies},
            addressline={Via Bonomea 265},
            city={Trieste},
            postcode={34136},
            country={Italy}}

\begin{abstract}
Any kind of network can be naturally represented by a directed acyclic graph; additionally, activation functions can determine the reaction of each node of the network with respect to the signal(s) incoming.
We study the characterization of the signal distribution in a network under the lens of tensor algebra. More specifically, we describe every activation function as tensor distributions with respect to the nodes, called \textit{activation tensors}. The distribution of the signal is encoded in the \textit{total tensor} of the network. We formally prove that the total tensor can be obtained by computing the \textit{Batthacharya-Mesner Product}, an $n$-ary operation for tensors of order $n$, on the set of activation tensors properly ordered and processed via two basic operations, that we call \textit{blow}  and \textit{forget}. Our theoretical framework can be validated through the related code developed in Python.
\end{abstract}


\begin{highlights}
\item Study of the characterization of the signal distribution in a network under the lens of tensor algebra
\item Formal proof of equality between the \textit{total tensor} of the network and the \textit{Batthacharya-Mesner Product} on the set of the activation tensors
\end{highlights}


\end{frontmatter}

\vskip1cm

 {\it \hfill The last shall be the first.} 
 
 {\it \hfill Matthew, 16:20}

\section*{Introduction}\label{sec:introduction}

In this work, we introduce a description of the behavior of networks which is based on the Batthacharya-Mesner Product (BMP) of tensors.
The BMP is an $n$-ary operation for tensors of order $n$, which generalizes the binary matrix multiplication.
It was introduced in~\cite{BM1} (see also~\cite{BM2} and~\cite{B}) and used for a description of algebraic and geometrical properties of tensors in~\cites{G, GF1, GF2, TKMP}. 

In our setting, a network is based on the underlying directed graph which describes the flow of a signal through a set of nodes. We study directed acyclic graphs (DAGs), which are graphs with no oriented cycles, to eliminate the circular transmissions of the signal. The absence of cycles guarantees the existence of total orderings of nodes compatible with the partial ordering defined by the graph. We choose a total ordering as a constituting aspect of the network.
In addition, we assign to each node an activation tensor which is responsible for the reaction of the node to the signals it receives. 

The order of the activation tensor corresponds to the number of incoming arrows ($in$-$degree$) of the node plus one. The dimensions of the activation tensor, in turn, are equal to the number of states of each node, which we assume to be constant through the network. In practice, for a node $a$ whose parents are the nodes $(b_{j_1},\dots,b_{j_d})$ (in the fixed  total order), the entry $A_{i_1,\dots, i_d,i_{d+1}}$ of the activation tensor $A$ measures the probability that $a$ assumes the state $i_{d+1}$ after receiving the signal $i_k$ from node $b_{j_k}$, for $k=1,\dots,d$. We consider the case in which all nodes transmit a signal equal to the state they assume and that no delays occur in communication among the nodes.

In our analysis, we transform the collection of tensors $\{A_i\}$, for each node $a_i$ of the graph, into tensors $\{T_{a_i}\}$ of order $n$, equal to the number of nodes in the graph, by two procedures on the indexes called {\it blow} (or {\it inflation}, following the terminology of~\cite{G}) and {\it forget} (or {\it copy}), described in Section~\ref{sec:bm_networks}.

We observe that these operations are valid regardless of the number of nodes and their activation tensors, which can be arbitrary and different for each node. Examples of reasonable activation tensors (following the setting of~\cite{BC}) are given throughout the paper. We observe that, to increase the correspondence with an interpretation of the network in statistics, one could require some stochastical restrictions in the activation tensors, e.g. the entries being positive reals in the interval $[0,1]$ and the marginalization relative to each state $i_{d+1}$ being equal to $1$. We decided not to introduce such restrictions because they are unnecessary for our computations.

It turns out that, for a fixed network with all nodes observable, the set of total tensors describing the final states of the network forms a toric variety.
Since in our network all nodes are observable, the total tensor of the network, which depends on parameters, fills a toric subvariety of the (projective) space of tensors. The main geometric aspects of this variety are therefore strictly related to the algebraic structure of the BMP, and we will explore the connection, as well as extensions to networks with hidden nodes, through marginalization of the BMP in future investigations.

Our main theorem (Theorem~\ref{main}) shows that it is possible to determine the final state of the network by computing the BMP, with a given ordering, of the activation tensors of its nodes suitably modified by the blow and forget operations introduced above.

Here is the structure of the paper. In Section~\ref{sec:notation}, we introduce the networks we want to analyze and define the notation and basic concepts on tensors, pointing out the notion of activation tensor. In Section~\ref{sec:bm_product}, we define the BMP of tensors with its algebraic properties, and finally, in Section~\ref{sec:bm_networks}, we show how it is possible to characterize a network based on the BMP.

\section{Basic notation and concepts}\label{sec:notation}

\subsection{Basic tensor notation}
In this section, we introduce the main notation for tensors and networks. A basic reference for tensor theory is~\cite{H, landsberg2011tensors}.
\begin{definition}  Let $\mathrm{K}$ be an algebraically closed field and $V_1, \dots, V_d$ vector spaces. A function
$$f: V_1 \times \cdots \times V_d \to \mathrm{K}$$
is \emph{multilinear} if it is linear in each factor $V_l$, $l = 1, \dots, d$. The space of such multilinear functions is denoted by $V_1^* \otimes V_2^* \otimes \cdots \otimes V_d^*$ and called the \emph{tensor product} of $V_1^*, \dots, V_k^*$. Elements $T \in V_1^* \otimes V_2^* \otimes \cdots \otimes V_d^*$ are called \emph{tensors}. The integer $d$ is sometimes called the \emph{order} of $T$. The sequence of natural numbers $(\mathbf{v}_1, \dots, \mathbf{v}_d)$ is called the \emph{dimension} of $T$.
We say a tensor is \emph{cubical} of order $n$ if $V_1 = \dots = V_d = n$.
\end{definition}
Our labeling notation for a tensor $T \in \C^2 \otimes \C^2 \otimes \C^2$ reported in~\Cref{fig:tensor_indices}.

\begin{figure}
    \centering
    \begin{tikzpicture}[scale=0.75,baseline={([yshift=-.5ex]current bounding box.center)}]
        \draw (0,0) node[scale=1] {100};
        \draw (2,0) node[scale=1] {110};
        \draw (1,1) node[scale=1] {101};
        \draw (3,1) node[scale=1] {111};
        \draw (0,2) node[scale=1] {000};
        \draw (2,2) node[scale=1] {010};
        \draw (1,3) node[scale=1] {001};
        \draw (3,3) node[scale=1] {011};      
        \draw (0,0.3) -- (0,1.7); 
        \draw (0.5,0) -- (1.5,0); 
        \draw (2,0.3) -- (2,1.7); 
        \draw (0.5,2) -- (1.5,2);        
        \draw (1,1.3) -- (1,2.7); 
        \draw (1.5,1) -- (2.5,1); 
        \draw (3,1.3) -- (3,2.7); 
        \draw (1.5,3) -- (2.5,3);       
        \draw (0.3,0.3) -- (0.7,0.7); 
        \draw (0.3,2.3) -- (0.7,2.7); 
        \draw (2.3,0.3) -- (2.7,0.7); 
        \draw (2.3,2.3) -- (2.7,2.7);      
    \end{tikzpicture}\label{tensor:1} 
    \caption{\textbf{Notation.} The labeling notation we consider for a tensor $T \in \C^2 \otimes \C^2 \otimes \C^2$.}
    \label{fig:tensor_indices}
\end{figure}

\subsection{Networks}

A basic reference for networks and activation tensors is~\cite{M}. Following the notation proposed in~\cite{DI}, let $G$ be a DAG with $V(G)$ the set of \textit{nodes}, or vertices, and $E(G)$ the set of \textit{edges}. We call {\it arrows} the oriented edges of $G$.
For any node $a \in V(G)$, we say that the node $b \in V(G)$ is a {\it parent} of $a$ if there is a directed arrow $b \to a$ in the graph, i.e. $b \to a \in E(G)$, and we define $parents(a) = \{ b \in V(G) : b \text{ is a parent of } a \}$. The \textit{in-degree} of a node $a$, \textit{in-deg}$(a)$, is the cardinality of $parents(a)$. The \textit{cardinality} of $G$ is denoted by $| V(G) | = q$.
\medskip

\begin{definition}\label{def:network}
In our setting, a \textit{network} $\Nc =\Nc (G, \mathcal{A})$ is defined by the following series of data:
\begin{enumerate}
\item[(\textit{i})] a directed graph $G=(V(G), E(G))$ without (oriented) cycles, whose nodes represent the discrete variables of the network; \label{enum:1}
\item[(\textit{ii})] a set $\mathcal{A}$ of activation tensors $A_i$, one for each node $a_i \in V(G)$, which determines the output given by the node, based on the inputs it receives;

\item[(\textit{iii})] an ordering of the nodes which is compatible with the partial ordering given by the graph (topological sort). \label{enum:3}
\end{enumerate}
\end{definition}

Note that the absence of cycles (as in~\Cref{enum:1}$(i)$) in the graph implies that there exists some total ordering of the nodes compatible with the partial ordering defined by the graph (see~\Cref{enum:3}$(iii)$). This can be easily seen by induction.
\medskip

We assume that all the variables associated to the nodes of the graph can assume the same $n$ states (numbered from $0$ to $n-1$). In this case we will talk about {\it $n$-ary networks}. Unless otherwise stated, in the upcoming sections we will consider binary networks.
\medskip

By~\Cref{def:network}, a network $\mathcal{N}(G)$ is defined by its underlying DAG and the tensors of appropriate size that are associated with each vertex.  These tensors are called activation tensors and, since they are not very standard, they are described in more detail below.

\begin{definition}\label{def:activation} Let $a_i$ be a node of an $n$-ary network $\Nc (G)$. Let $d=\text{\it in-deg}(a_i)$. The \emph{activation tensor} $A_i \in \mathcal{A}$ that the network associates to $a_i$ is any tensor of order $d+1$ such that
$$A_i \in \underbrace{(\C^n)\otimes\cdots\otimes (\C^n)}_{d+1 \textit{ times}}.$$
\end{definition}


In particular, if $d = 0$, then an activation tensor at $a_i$ is simply a vector of $\C^n$. If $d = 1$, then an activation tensor at $a_i$ is an $n\times n $ matrix of complex numbers.

In the binary case, we will assume that node $a_i$ is inactive when its output is $0$ and $a_i$ is activated if its output is $1$.

\subsection{Tensors of networks}

Let us explain the role of the activation tensors of the network.

Consider a $n$-ary signal transmitted by the nodes of the network $\mathcal{N}$, following the pattern given by the arrows, \textit{i.e.} the directed edges. Let $d$ be the in-degree of $a$ and $\{ b_1,\dots,b_d \}$ the set of its parents, ordered accoding to the total ordering of the network. Then, node $a$ decides its state depending on the signals it receives, which it transmits to all its outgoing arrows.
\smallskip

The entry $A_{i_1,\dots,i_d,i_{d+1}}$ of the activation tensor represents the \emph{frequency} with which the node $a$ assumes the state $i_{d+1}$ after receiving the signal $i_1$ from $b_1$, the signal $i_2$ from $b_2$, \dots, and the signal $i_d$ from $b_d$.

\remark{We could take the stochastic point of view, by requiring that $\sum_{l=0}^{n-1} A_{i_1,\dots,i_d,l}=1$, so that $A_{i_1,\dots,i_d,i_{d+1}}$ represents the \emph{probability} that the node $a$ assumes the state $i_{d+1}$ after receiving the signal $i_1$ from $b_1$, \dots, and the signal $i_d$ from $b_d$. We choose not to adopt this point of view since the assumption is not necessary in our analysis.}

\begin{example}\label{qt1}
Some examples of activation tensors, for different values of $d$, are listed in the following.
\begin{itemize}
    \item When $d=0$,  the activation tensor $A=(a_0,\dots,a_{n-1})$ is defined such that node $a$ assumes the state $i$ (and casts the signal $i$) with frequency $a_i$.
    \item When $d=1$, well-known examples of activation tensors are the \emph{Jukes-Cantor} matrices, in which all the entries in the diagonal are equal to a parameter $\alpha$ and all the entries off the diagonal are equal to a parameter $\beta$. E.g., in the binary case,
\[
A= \begin{pmatrix} \alpha &\beta \\ \beta & \alpha \end{pmatrix},
\] with $\alpha, \beta \in \mathbb{C}$.


\item For $d\geq 1$, a relevant case of activation tensor for binary networks is the \emph{threshold of value one}, or the tensor associated to the Boolean logical operator $\exists$. The node $a$ assumes the value $1$ if at least one of the parents casts $1$, while $a$ assumes the state $0$ otherwise. Hence $A_{i_1,\dots,i_d,i_{d+1}}$ is $0$ if $i_1=\dots=i_d=0$ and $i_{d+1}=1$, or when at least one index $i_j$, $j=0,\dots, d$, is $1$ and $i_{d+1}=0$. Then, the activation tensor $A$ assumes the same value $1$ in all the remaining cases.
An example for  $d=2$, using the index convention as the tensor in~\Cref{fig:tensor_indices}, is represented in~\Cref{fig:threshold_1}.
This activation tensor can be extended to more general, $n$-ary networks.

\begin{figure}
    \centering
    \begin{tikzpicture}[scale=0.75,baseline={([yshift=-.5ex]current bounding box.center)}]

    \draw (0,0) node[scale=1] {$0$};
    \draw (2,0) node[scale=1] {$0$};
    \draw (1,1) node[scale=1] {$1$};
    \draw (3,1) node[scale=1] {$1$};
    \draw (0,2) node[scale=1] {$1$};
    \draw (2,2) node[scale=1] {$0$};
    \draw (1,3) node[scale=1] {$0$};
    \draw (3,3) node[scale=1] {$1$};
    
    \draw (0,0.3) -- (0,1.7); 
    \draw (0.3,0) -- (1.7,0); 
    \draw (2,0.3) -- (2,1.7); 
    \draw (0.3,2) -- (1.7,2); 
    
    \draw (1,1.3) -- (1,2.7); 
    \draw (1.3,1) -- (2.7,1); 
    \draw (3,1.3) -- (3,2.7); 
    \draw (1.3,3) -- (2.7,3);
    
    \draw (0.3,0.3) -- (0.7,0.7); 
    \draw (0.3,2.3) -- (0.7,2.7); 
    \draw (2.3,0.3) -- (2.7,0.7); 
    \draw (2.3,2.3) -- (2.7,2.7); 
    \end{tikzpicture}
    \caption{\textbf{Activation tensor with threshold of value one.} Example of activation tensor of threshold of value one for a node in a binary network, with $d = 2$.}
    \label{fig:threshold_1}
\end{figure}

\item The activation tensor above introduced can be modified, replacing the output $0,1$ with two parameters $\alpha, \beta \in \mathbb{C}$, yielding the \emph{quantum}-\emph{threshold of value one}. Specifically, the node $a$ will assume value $\alpha$ if it obeys the rule, and value $\beta$ elsewise.

 For $d=2$ the tensor (see~\Cref{fig:quantum_threshold}) is obtained by modifying the $0$'s with $\beta$'s and the $1$'s with $\alpha$'s in the tensor of~\Cref{fig:threshold_1}.

\begin{figure}
    \centering
    \begin{tikzpicture}[scale=0.75,baseline={([yshift=-.5ex]current bounding box.center)}] 

    \draw (0,0) node[scale=1] {$\beta$};
    \draw (2,0) node[scale=1] {$\beta$};
    \draw (1,1) node[scale=1] {$\alpha$};
    \draw (3,1) node[scale=1] {$\alpha$};
    \draw (0,2) node[scale=1] {$\alpha$};
    \draw (2,2) node[scale=1] {$\beta$};
    \draw (1,3) node[scale=1] {$\beta$};
    \draw (3,3) node[scale=1] {$\alpha$};
    
    \draw (0,0.3) -- (0,1.7); 
    \draw (0.3,0) -- (1.7,0); 
    \draw (2,0.3) -- (2,1.7); 
    \draw (0.3,2) -- (1.7,2); 
    
    \draw (1,1.3) -- (1,2.7); 
    \draw (1.3,1) -- (2.7,1); 
    \draw (3,1.3) -- (3,2.7); 
    \draw (1.3,3) -- (2.7,3);
    
    \draw (0.3,0.3) -- (0.7,0.7); 
    \draw (0.3,2.3) -- (0.7,2.7); 
    \draw (2.3,0.3) -- (2.7,0.7); 
    \draw (2.3,2.3) -- (2.7,2.7); 
    \end{tikzpicture}
    \caption{\textbf{Activation tensor with quantum-threshold of value one.} Example of activation tensor of threshold of value one for a node in a binary network, with $d = 2$, and parameters $\alpha, \beta \in \mathbb{C}$.}
    \label{fig:quantum_threshold}
\end{figure}
\end{itemize}
\end{example}

\begin{definition}\label{def:total} In a network $\Nc(G, \mathcal{A})$, with $\left | V(G) \right | = q$ and nodes $a_1,\dots, a_q$, the \emph{total tensor} associated to the network is the order-$q$ tensor $N$ of size $\underbrace{n \times \cdots \times n}_{q\textit{ times}}$ whose entry $N_{i_1,\dots,i_q}$, where $i_h = 0, \dots, n-1$ for $h=1, \dots, q$, is defined  as follows. For all nodes $a_j$, $j = 1, \dots, q$, $\text{\it in-deg}(a_j) = d_j<q$, let $A_j$ be the activation tensor and $P_j$  be the $d_j$-uple of indices in $\{i_1,\dots,i_q\}$ corresponding to the parents of $a_j$, \textit{i.e.} $P_j = 
( i_k :  b_k \in parents(a_j) )$ (where the tuple is ordered according to the total ordering of the network). Then

\begin{equation}\label{eq:total_tensor}
N_{i_1,\dots,i_q}= \Pi_{j=1}^q (A_j)_{\mathsf{concat}[P_j,i_j]}.
\end{equation} \label{def:total_tensor}
\end{definition}

where $\mathsf{concat}[x,y]$ is the concatenation of tuples $x$ and $y$.

The following examples, in which we assume the network is binary, \textit{i.e.} only signals $0$ and $1$ can be transmitted between the nodes, illustrate the meaning of the total tensor. 

\begin{example}
Consider the network in~\Cref{fig:Markov_network} with three nodes $(a, b, c)$ (in the unique possible ordering). It corresponds to a short Markov chain in which $q = 3$. \\We endow the nodes with the following activation tensors: the vector $A = (\alpha,\beta)$ for the source $a$, the Jukes-Cantor matrix $B = \begin{pmatrix} \alpha & \beta \\ \beta & \alpha \end{pmatrix}$ for $b$ and set $C = B$ for $c$.

\smallskip
\begin{figure}[!htt]
    \centering
    \begin{tikzpicture}[
            > = stealth, 
            shorten > = 1pt, 
            auto,
            node distance = 3cm, 
            semithick 
        ]
        \tikzstyle{every state}=[
            draw = black,
            thick,
            fill = white,
            minimum size = 12mm
        ]
        \node[state] (1) {$a$};
        \node[state] (2) [right of=1] {$b$};
        \node[state] (3) [right of=2] {$c$}; 

        \path[->] (1) edge node {} (2);
        \path[->] (2) edge node {} (3); 
    \end{tikzpicture}
    \caption{\textbf{Markov chain with three nodes.} Tensor network with three nodes $a, b, c$. Node $a$ transmits a binary signal to node $b$ which, in turn, sends a signal to node $a$.}
    \label{fig:Markov_network}
\end{figure}
In order to compute the total tensor, we determine the indices of the parents of each node.
Let $a, b, c$ be renamed to $a_1, a_2, a_3$, respectively. For node $a_1$, which is the source, $\textit{in-deg}(a_1) = 0$ and $P_1 = ()$.
For node $a_2$, we have $\textit{in-deg}(a_2) = 1$, $parents(a_2) = \{a_1\}$ and $P_2 = (i_1)$.
Finally, for node $a_3$, we have $\textit{in-deg}(a_3) = 1$, $parents(a_3) = \{a_2\}$ and $P_3 = (i_2)$.

Then, the total tensor $N$ of the Markov chain with three nodes, according to~\Cref{def:total_tensor}, can be computed as
\begin{equation*}\label{eq:total_tensor_Markov}
N_{i_1, i_2, i_3} = A_{i_1} B_{i_1, i_2} C_{i_2, i_3},
\end{equation*}
and represented as in~\Cref{fig:total_tensor_markov}.

\begin{figure}
    \centering
    \begin{tikzpicture}[scale=0.75,baseline={([yshift=-.5ex]current bounding box.center)}] 
    \draw (0,0) node[scale=1] {$\alpha\beta^2$};
    \draw (2,0) node[scale=1] {$\alpha\beta^2$};
    \draw (1,1) node[scale=1] {$\beta^3$};
    \draw (3,1) node[scale=1] {$\alpha^2\beta$};
    \draw (0,2) node[scale=1] {$\alpha^3$};
    \draw (2,2) node[scale=1] {$\alpha\beta^2$};
    \draw (1,3) node[scale=1] {$\alpha^2\beta$};
    \draw (3,3) node[scale=1] {$\alpha^2\beta$};
    
    \draw (0,0.3) -- (0,1.7); 
    \draw (0.5,0) -- (1.5,0); 
    \draw (2,0.3) -- (2,1.7); 
    \draw (0.4,2) -- (1.5,2); 
    
    \draw (1,1.3) -- (1,2.7); 
    \draw (1.3,1) -- (2.5,1); 
    \draw (3,1.3) -- (3,2.7); 
    \draw (1.5,3) -- (2.5,3);
    
    \draw (0.3,0.3) -- (0.7,0.7); 
    \draw (0.3,2.3) -- (0.7,2.7); 
    \draw (2.3,0.3) -- (2.7,0.7); 
    \draw (2.3,2.3) -- (2.7,2.7); 
    \end{tikzpicture} 
    \caption{\textbf{Total tensor of the Markov chain with three nodes.} Representation of the total network of a binary Markov chain with three nodes $a, b, c$ and parameters $\alpha, \beta\in \mathbb{C}$. A possible reading of the values of the  total tensor is the following:  $N_{000}$ means that $a$ casts $0$ ($\alpha$), then $b$ obeys and casts $0$ ($\alpha$ again), and $c$ obeys and takes the $0$ state (another $\alpha$), so $N_{000}=\alpha^3$; $N_{011}$ means that $a$ casts $0$ ($\alpha$), then $b$ disobeys and casts $1$ ($\beta$), while $c$, which receives  $1$, obeys and takes the $1$ state (another $\alpha$), so $N_{011}=\alpha^2\beta$; and so on.}
    \label{fig:total_tensor_markov}
\end{figure}
\end{example}

\begin{example}
Consider the network with three nodes $(a, b, c)$ (in the unique possible ordering) associated to the graph in Fig.\ref{fig:triangle_network}. We endow the nodes with the following activation tensors: the vector $(\alpha,\beta)$ for the source $a$, the Jukes-Cantor matrix $B = \begin{pmatrix} \alpha & \beta \\ \beta & \alpha \end{pmatrix}$ for $b$ and the quantum-threshold of value one tensor $C$ of Example~\ref{qt1} for $c$.
\begin{figure}[!htt]
    \centering
    \begin{tikzpicture}[
            > = stealth, 
            shorten > = 1pt, 
            auto,
            node distance = 3cm, 
            semithick 
        ]
        \tikzstyle{every state}=[
            draw = black,
            thick,
            fill = white,
            minimum size = 12mm
        ]
        \node[state] (1) {$a$};
        \node[state] (2) [below left of=1] {$b$};
        \node[state] (3) [below right of=1] {$c$}; 

        \path[->] (1) edge node {} (2);
        \path[->] (1) edge node {} (3);
        \path[->] (2) edge node {} (3); 
    \end{tikzpicture}
    \caption{\textbf{Acyclic triangle with three nodes.} Tensor network with three nodes $a, b, c$ and their associated activation tensors $A, B, C$. Node $a$ transmits a binary signal to node $b$ and node $c$. In turn, node $b$ sends a signal to node $c$.}
    \label{fig:triangle_network}
\end{figure}
\smallskip

For node $a = b_1$, which is the source, $A = (\alpha, \beta)$, $\textit{in-deg}(a) = 0$ and $P_1 = ()$. 
For node $b = b_2$, we have $\textit{in-deg}(b) = 1$ and $P_2 = (i_1)$.
For node $c = b_3$, we have $\textit{in-deg}(c) = 2$ and $P_3 = (i_1, i_2)$.

\smallskip

Then, the total tensor $N$ of the network, according to Equation~\eqref{eq:total_tensor} can be computed as
\begin{equation*}
N_{i_1, i_2, i_3} = A_{i_1} B_{i_1, i_2} C_{i_1, i_2, i_3},
\end{equation*}
with computations made explicit in~\Cref{fig:total_triangle_network}.

\begin{figure}[!htt]
    \centering
    \begin{tikzpicture}[scale=0.75,baseline={([yshift=-.5ex]current bounding box.center)}] 
    
    \draw (0,0) node[scale=1] {$\beta^3$};
    \draw (2,0) node[scale=1] {$\alpha\beta^2$};
    \draw (1,1) node[scale=1] {$\alpha\beta^2$};
    \draw (3,1) node[scale=1] {$\alpha^2\beta$};
    \draw (0,2) node[scale=1] {$\alpha^3$};
    \draw (2,2) node[scale=1] {$\alpha\beta^2$};
    \draw (1,3) node[scale=1] {$\alpha^2\beta$};
    \draw (3,3) node[scale=1] {$\alpha^2\beta$};
    
    \draw (0,0.3) -- (0,1.7); 
    \draw (0.3,0) -- (1.5,0); 
    \draw (2,0.3) -- (2,1.7); 
    \draw (0.3,2) -- (1.5,2); 
    
    \draw (1,1.3) -- (1,2.7); 
    \draw (1.4,1) -- (2.5,1); 
    \draw (3,1.3) -- (3,2.7); 
    \draw (1.5,3) -- (2.5,3);
    
    \draw (0.3,0.3) -- (0.7,0.7); 
    \draw (0.3,2.3) -- (0.7,2.7); 
    \draw (2.3,0.3) -- (2.7,0.7); 
    \draw (2.3,2.3) -- (2.7,2.7); 
    \end{tikzpicture}
    \caption{\textbf{Total tensor of the acyclic triangle.} Representation of the total network of a binary acyclic triangle with nodes $a, b, c$ and parameters $\alpha, \beta\in \mathbb{C}$.}
    \label{fig:total_triangle_network}
    \end{figure}

\end{example}

\section{The Batthacharya-Mesner product}\label{sec:bm_product}

\subsection{Batthacharya-Mesner Product}
The Batthacharya-Mesner Product of tensors is an $n$-ary operation in the  space of order-$n$ cubical tensors which generalizes the row-by-column product of matrices. It is defined for tensors with coefficients in any field $\K$ (and even when $\K$ is just a ring), although we will apply it mainly to complex tensors.

\begin{definition} \label{BMP}
    Let $T_1, T_2, \dots, T_d$ be order-$d$ tensors of sizes $n_1 \times l \times n_3 \times \dots \times n_d$, $n_1 \times n_2 \times l \times \dots \times n_d$, $\cdots, l \times n_2 \times n_3 \times \dots \times n_d $, respectively. The \textit{generalized Bhattacharya-Mesner Product}  (shortly, the \textit{BMP}), denoted as $\circ (T_1, T_2, \dots, T_d) $ (see~\cite{G}), is a $d$-ary operation whose result is a tensor $T$ of size $n_1 \times n_2 \times \cdots \times n_d$ such that:
    \begin{equation*}
        (T)_{i_1\dots i_d} = \sum \limits_{h=0}^{l-1} (T_d)_{h i_2 \dots i_d} (T_1)_{i_1 h \dots i_d} \cdots (T_{d-1})_{i_1 i_2 \dots h}.
    \end{equation*}
\end{definition}

\begin{remark}
    The BMP generalizes the classical matrix product when $d=2$. Indeed, if we consider two matrices $A_1 \in \mathbb{K}^{n_1 \times l}$, $A_2 \in \mathbb{K}^{l \times n_2}$, their product is defined as 
    \begin{equation*}
        (A_1 A_2)_{ij} = \sum \limits_{h=0}^{l-1} (A_1)_{ih} (A_2)_{hj} = \sum \limits_{h=0}^{l-1} (A_2)_{hj} (A_1)_{ih},
    \end{equation*}
    which corresponds to the BMP when the order is equal to $2$ and clarifies the inversion of the tensors in the product.
    \end{remark}

\begin{example} Let us give an example of the product of three $2\times 2\times 2$ tensors.
\begin{center}
\begin{equation*}
\begin{tikzpicture}[scale=0.7,baseline={([yshift=-.5ex]current bounding box.center)}]  
\draw (-1 ,1.5) node[scale=1.3] {$\circ$};
\draw (-0.5 ,1.5) node[scale=2] {$($};
\draw (0,0) node[scale=0.9] {$5$};
\draw (2,0) node[scale=0.9] {$6$};
\draw (1,1) node[scale=0.9] {$7$};
\draw (3,1) node[scale=0.9] {$8$};
\draw (0,2) node[scale=0.9] {$1$};
\draw (2,2) node[scale=0.9] {$2$};
\draw (1,3) node[scale=0.9] {$3$};
\draw (3,3) node[scale=0.9] {$4$};
\draw (0,0.3) -- (0,1.7); 
\draw (0.3,0) -- (1.7,0); 
\draw (2,0.3) -- (2,1.7); 
\draw (0.3,2) -- (1.7,2); 
\draw (1,1.3) -- (1,2.7); 
\draw (1.3,1) -- (2.7,1); 
\draw (3,1.3) -- (3,2.7); 
\draw (1.3,3) -- (2.7,3);
\draw (0.3,0.3) -- (0.7,0.7); 
\draw (0.3,2.3) -- (0.7,2.7); 
\draw (2.3,0.3) -- (2.7,0.7); 
\draw (2.3,2.3) -- (2.7,2.7); 

\draw (3.5 ,1) node[scale=2] {$,$};

\draw (4,0) node[scale=0.9] {$13$};
\draw (6,0) node[scale=0.9] {$14$};
\draw (5,1) node[scale=0.9] {$15$};
\draw (7,1) node[scale=0.9] {$16$};
\draw (4,2) node[scale=0.9] {$9$};
\draw (6,2) node[scale=0.9] {$10$};
\draw (5,3) node[scale=0.9] {$11$};
\draw (7,3) node[scale=0.9] {$12$};
\draw (4,0.3) -- (4,1.7); 
\draw (4.3,0) -- (5.7,0); 
\draw (6,0.3) -- (6,1.7); 
\draw (4.3,2) -- (5.7,2);
\draw (5,1.3) -- (5,2.7); 
\draw (5.3,1) -- (6.7,1); 
\draw (7,1.3) -- (7,2.7); 
\draw (5.3,3) -- (6.7,3);
\draw (4.3,0.3) -- (4.7,0.7); 
\draw (4.3,2.3) -- (4.7,2.7); 
\draw (6.3,0.3) -- (6.7,0.7); 
\draw (6.3,2.3) -- (6.7,2.7); 

\draw (7.5 ,1) node[scale=2] {$,$};

\draw (8,0) node[scale=0.9] {$21$};
\draw (10,0) node[scale=0.9] {$22$};
\draw (9,1) node[scale=0.9] {$23$};
\draw (11,1) node[scale=0.9] {$24$};
\draw (8,2) node[scale=0.9] {$17$};
\draw (10,2) node[scale=0.9] {$18$};
\draw (9,3) node[scale=0.9] {$19$};
\draw (11,3) node[scale=0.9] {$20$};
\draw (8,0.3) -- (8,1.7); 
\draw (8.3,0) -- (9.7,0); 
\draw (10,0.3) -- (10,1.7); 
\draw (8.3,2) -- (9.7,2);
\draw (9,1.3) -- (9,2.7); 
\draw (9.3,1) -- (10.7,1); 
\draw (11,1.3) -- (11,2.7); 
\draw (9.3,3) -- (10.7,3);
\draw (8.3,0.3) -- (8.7,0.7); 
\draw (8.3,2.3) -- (8.7,2.7); 
\draw (10.3,0.3) -- (10.7,0.7); 
\draw (10.3,2.3) -- (10.7,2.7); 
\draw (11.7 ,1.5) node[scale=2] {$)$};
\draw (12.5,1.5) node[scale=1] {$=$};

\draw (14,0) node[scale=0.7] {$2995$};
\draw (16,0) node[scale=0.7] {$3372$};
\draw (15,1) node[scale=0.7] {$4489$};
\draw (17,1) node[scale=0.7] {$5032$};
\draw (14,2) node[scale=0.7] {$615$};
\draw (16,2) node[scale=0.7] {$708$};
\draw (15,3) node[scale=0.7] {$1525$};
\draw (17,3) node[scale=0.7] {$1752$};
\draw (14,0.3) -- (14,1.7); 
\draw (14.5,0) -- (15.5,0); 
\draw (16,0.3) -- (16,1.7); 
\draw (14.5,2) -- (15.5,2);
\draw (15,1.3) -- (15,2.7); 
\draw (15.5,1) -- (16.5,1); 
\draw (17,1.3) -- (17,2.7); 
\draw (15.5,3) -- (16.5,3);
\draw (14.3,0.3) -- (14.7,0.7); 
\draw (14.3,2.3) -- (14.7,2.7); 
\draw (16.3,0.3) -- (16.7,0.7); 
\draw (16.3,2.3) -- (16.7,2.7); 

\end{tikzpicture}
\end{equation*}
\end{center}
\end{example}

\smallskip

\subsection{Some algebraic properties of BMP}

It is almost immediate to realize that the BMP is non-commutative. The BMP is also non-associative, in the sense that e.g. for general tensors $A,B,C,D,E$ of order $3$
$$ \circ(\circ(A,B,C),D,E) \neq \circ(A,B,\circ(C,D,E)).$$
Yet, the BMP is multilinear, and satisfies some interesting properties.

\vspace{3mm}
{\bf Multiple identities}.
We keep our comparison with matrices ($=$ order-$2$ tensors) to show how their behavior is generalized in case of order-$d$ tensors, with $d>2$. Indeed, in the case $d = 2$, for any matrix $A$ of size $m \times n$, there exist identity matrices $I_m$ and $I_n$ of size $m \times m$ and $n \times n$, respectively, such that $AI_n=A$ and $I_mA=A$. The role of the identity matrix can be generalized for order-$n$ tensors by introducing \textit{multi-identities}.

\begin{definition} For all $j,k= 1,\dots,d$ with $j<k$, define the order-$d$ \emph{identitary} tensors $I^{j \ k}$ by
$$ (I^{j \ k})_{i_1,\dots,i_d}=\begin{cases} 1 & \text{ if } i_j=i_k; \\ 0 &\text{ else.} \end{cases}$$
\end{definition}

Notice that when $d=2$ we have only one of these tensors, namely $I^{1 \ 2}$, which corresponds to the identity matrix.
\smallskip

\begin{example}\label{ex:multi} In the case of $2 \times 2 \times 2$ tensors, the order-$3$ identitary tensors are the following:
\begin{center}
\begin{equation}
\begin{tikzpicture}[scale=0.75,baseline={([yshift=-.5ex]current bounding box.center)}] 
\draw (-1,1.3) node[scale=0.9] {$I^{1 \ 2}=$};
\draw (0,0) node[scale=0.9] {$0$};
\draw (2,0) node[scale=0.9] {$1$};
\draw (1,1) node[scale=0.9] {$0$};
\draw (3,1) node[scale=0.9] {$1$};
\draw (0,2) node[scale=0.9] {$1$};
\draw (2,2) node[scale=0.9] {$0$};
\draw (1,3) node[scale=0.9] {$1$};
\draw (3,3) node[scale=0.9] {$0$};
\draw (0,0.3) -- (0,1.7); 
\draw (0.3,0) -- (1.7,0); 
\draw (2,0.3) -- (2,1.7); 
\draw (0.3,2) -- (1.7,2);
\draw (1,1.3) -- (1,2.7); 
\draw (1.3,1) -- (2.7,1); 
\draw (3,1.3) -- (3,2.7); 
\draw (1.3,3) -- (2.7,3);
\draw (0.3,0.3) -- (0.7,0.7); 
\draw (0.3,2.3) -- (0.7,2.7); 
\draw (2.3,0.3) -- (2.7,0.7); 
\draw (2.3,2.3) -- (2.7,2.7);

\draw (4,1.3) node[scale=0.9] {$,\  I^{1 \ 3}=$};
\draw (5,0) node[scale=0.9] {$0$};
\draw (7,0) node[scale=0.9] {$0$};
\draw (6,1) node[scale=0.9] {$1$};
\draw (8,1) node[scale=0.9] {$1$};
\draw (5,2) node[scale=0.9] {$1$};
\draw (7,2) node[scale=0.9] {$1$};
\draw (6,3) node[scale=0.9] {$0$};
\draw (8,3) node[scale=0.9] {$0$};
\draw (5,0.3) -- (5,1.7); 
\draw (5.3,0) -- (6.7,0); 
\draw (7,0.3) -- (7,1.7); 
\draw (5.3,2) -- (6.7,2);
\draw (6,1.3) -- (6,2.7); 
\draw (6.3,1) -- (7.7,1); 
\draw (8,1.3) -- (8,2.7); 
\draw (6.3,3) -- (7.7,3);
\draw (5.3,0.3) -- (5.7,0.7); 
\draw (5.3,2.3) -- (5.7,2.7); 
\draw (7.3,0.3) -- (7.7,0.7); 
\draw (7.3,2.3) -- (7.7,2.7);

\end{tikzpicture}.
\end{equation}
\end{center}
\end{example}

The role of identitary tensors come out from the following.

\begin{proposition} \label{prop:multi}
For all tensors $A$ of order $d$ and for any $j=1,\dots,d$, it holds that the BMP
$$ \circ (I^{1 \ 2}, \dots , I^{1 \ d}, A) =A \quad $$
and, for $j > 1$,
$$ \circ (I^{2 \ j}, \dots , I^{j-1\ j},  A , I^{j\ j+1}, \dots, I^{j \ d},  I^{1 \ j}) =A.$$

\end{proposition}
\begin{proof} It is a simple computation. Call $T$ the product. Then, for $j > 1$:
\begin{align*} T_{i_1,\dots,i_d} = \sum_{h=0}^{l-1} (I^{1 \ j})_{h,i_2,\dots, i_d}\cdots (I^{j-1\ j})_{i_1,\dots,h,\dots, i_d} \cdot \\ \cdot 
A_{i_1,\dots ,h,\dots, i_d}\cdots (I^{j \ j+1})_{i_1,\dots,h,\dots, i_d} \cdots (I^{j \ d})_{i_1, \dots i_{d-1},h}.   \end{align*}

The first factor in any summand $ (I^{1 \ j})_{l,i_2,\dots, i_d}$  implies that only the summand with $l=i_j$ is different from 0. Since
$$(I^{1 \ j})_{i_j,i_2,\dots, i_d}= \dots = (I^{j \ d})_{i_1,\dots, i_{d-1},i_j} = 1,$$
then $ T_{i_1,\dots,i_d}=A_{i_1,\dots,i_d}$.

When $j=1$, the same conclusion holds by looking at the last factor of each summand.
\end{proof}

According to~\Cref{prop:multi}, it is easy to show that the order-$3$ identitary tensors of~\Cref{ex:multi} satisfy $\circ(I^{1 \ 2}, I^{\, 3}, A) = A$ for each $2 \times 2 \times 2$ tensor $A$.

\vspace{3mm}
{\bf Transposition.}
In the case of order-$2$ tensors, the transpose of the product of two matrices is equal to the product of the transpose of the second matrix and the transpose of the first matrix, i.e. for matrices $A,B$ we have $(AB)^T = B^TA^T$. Again, the behavior of the matrix transposition can be generalized for order-$d$ tensors.

\begin{definition}
    Let $A$ be an order-$d$ tensor. A \textit{$\sigma$-transpose of $A$} is an order-$d$ tensor $A^{\sigma}$ obtained by permuting with $\sigma$ the $d$ indices of $A$.
\end{definition}

In case of matrices we have $2$ indices subject to $2$ permutations, obtaining a unique transpose of a matrix; in case of order-$d$ tensors we have $d$ indices subject to $d!$ permutations, obtaining $(d!-1)$ $\sigma$-transposes or an order-$d$ tensor.

From the definition of the BMP, it is immediate to prove the following.

\begin{proposition}
    Let $A_1, \dots, A_d$ be tensors of order $d$ and, given $\sigma$ a permutation of $d$ indices, call $A_1^{\sigma}, \dots, A_d^{\sigma}$ their $\sigma$-transposes. Then
    \begin{equation*}
         \circ (A_1, \dots , A_d)^{\sigma} =   \circ( A_{\sigma(1)}^{\sigma}, \dots , A_{\sigma(d)}^{\sigma}).
    \end{equation*}
\end{proposition}

\begin{example} 
We report the computation for the case $d=3$. \\ 
Let $A,B,C$ be $3$-tensors and consider a permutation $\sigma: \{ 1,2,3 \} \rightarrow \{ 1,2,3 \}$ such that $\sigma(1) = 1,$ $\sigma(2) = 3,$  $\sigma(3) = 2$. Define $T^{\sigma} = (\circ(A,B,C))^{\sigma}$ and $T' = \circ((A,B,C)^\sigma)$. Then

\begin{equation}
     (T_{ijk})^{\sigma} = T_{ikj} = A_{i0j} B_{ik0} C_{0kj} + A_{i1j} B_{ik1} C_{1kj}, 
\end{equation}

and
\begin{equation}
T'_{ijk}= \circ((A,B,C)^\sigma)_{ijk} = \circ(A^{\sigma}, B^{\sigma}, C^{\sigma})_{ijk} = A'_{i0k} B'_{ij0} C'_{0jk} + A'_{i1k} B'_{ij1} C'_{1jk}, 
\end{equation}
with $A'_{ijk}= A_{ikj}, B'_{ijk}= B_{ikj}, C'_{ijk}= C_{ikj}.$ It follows that

\begin{equation}
   \circ(A^{\sigma}, B^{\sigma}, C^{\sigma})_{ijk} = T'_{ijk} = A_{i0j} B_{ik0} C_{0kj} + A_{i1j} B_{ik1} C_{1kj} = T_{ikj} = (\circ(A,B,C)_{ijk})^{\sigma}.
\end{equation}

\end{example}
\section{Network characterization in terms of BMP}\label{sec:bm_networks}

\subsection{Expansion of tensors}

The following operations are fundamental to transform the activation tensor of each node in a tensor which gives, via the BMP, the total tensor of the network.

We call the operations \emph{blow} and \emph{forget}. The former expands a tensor of order $d$ to a tensor of order $d+1$, while the latter expands a tensor of order $d$ to a tensor of generic order $d'>d$.

\vspace{3mm}
\paragraph{\bf The \textit{blow} operation}

\begin{definition}\label{def:blow}
Let $T\in V^{\otimes d}$. The \emph{blow}, or \emph{inflation}, of $T$ is the tensor $b(T)$ of order $d+1$ defined by

\begin{equation}
(b(T))_{i_1, \dots, i_{d+1}} = \bigg \{
\begin{array}{ll}
T_{i_1, \dots, i_d}& \text{if} \ i_1 = i_{d+1}; \\
0 & \text{else.} \\
\end{array}
\end{equation}
\end{definition}

We observe that our definition of inflation (blow) extends the notion introduced in~\cite[Section 2.4]{G}.

\begin{example} The blow $b(M)$ of a matrix $M=\begin{pmatrix} a & b \\ c & d \end{pmatrix}$ giving an order-$3$ tensor is the following:

\begin{center}
\begin{equation}
\begin{tikzpicture}[scale=0.75,baseline={([yshift=-.5ex]current bounding box.center)}] 

\draw(-1.2,1.3) node[scale=1] {$b(M) =$};
\draw (0,0) node[scale=1] {$0$};
\draw (2,0) node[scale=1] {$0$};
\draw (1,1) node[scale=1] {$c$};
\draw (3,1) node[scale=1] {$d$};
\draw (0,2) node[scale=1] {$a$};
\draw (2,2) node[scale=1] {$b$};
\draw (1,3) node[scale=1] {$0$};
\draw (3,3) node[scale=1] {$0$};

\draw (0,0.3) -- (0,1.7); 
\draw (0.3,0) -- (1.7,0); 
\draw (2,0.3) -- (2,1.7); 
\draw (0.3,2) -- (1.7,2); 

\draw (1,1.3) -- (1,2.7); 
\draw (1.3,1) -- (2.7,1); 
\draw (3,1.3) -- (3,2.7); 
\draw (1.3,3) -- (2.7,3);

\draw (0.3,0.3) -- (0.7,0.7); 
\draw (0.3,2.3) -- (0.7,2.7); 
\draw (2.3,0.3) -- (2.7,0.7); 
\draw (2.3,2.3) -- (2.7,2.7); 
\end{tikzpicture}
\end{equation}
\end{center}
\end{example}

Notice that the blow is defined only in coordinates. 

\begin{remark} If $T=v_1\otimes\cdots\otimes v_d$ is a tensor of rank $1$, then its blow (in general) is no longer of rank $1$. In term of a basis $\{e_1,\dots,e_n\}$ of $V$, for which $v_1=a_1e_1+\dots+a_ne_n$, the blow $b(T)$ is the tensor

\begin{equation*} 
        b(T)= \sum_{i = 1}^n a_ie_i\otimes v_2\otimes \cdots\otimes v_d\otimes e_i.
\end{equation*}
\end{remark}

\vspace{3mm}
{\bf The \textit{forget} operation.}

\begin{definition}\label{def:forget}
Let $T\in V^{\otimes d}$. For $d'>d$ let $J\subset \{1,\dots, d' \}$, $|J| = d'-d$.
The $J$-th \emph{forget}, or $J$-th \emph{copy} of $T$ is the tensor $f_J(T)$ of order $d'$ defined by

\begin{equation}
(f_J(T))_{i_1, \dots, i_{d'}} = T_{j_1, \dots,j_{d}}
\end{equation}      
where $(j_1, \dots,j_{d})$ is the list obtained from $(i_1, \dots, i_{d'})$ by erasing the element $i_k$ for all $k\in J$.
\end{definition}

For instance, the forget operation $f_{\{k\}}(T)$ creates a tensor whose entry of indexes $i_1,\dots,i_{d+1}$ corresponds to the entry of $T$ where $i_h = j_h$ for $h=1, \dots, k-1$, $i_h = j_{h-1}$ for $h=k+1, \dots , d+1$.

\begin{example}The forget $f_{\{2\}}(M)$ of a matrix $M=\begin{pmatrix} a & b \\ c & d \end{pmatrix}$ giving an order-$3$ tensor is the following:

\begin{center}
\begin{equation}
\begin{tikzpicture}[scale=0.75,baseline={([yshift=-.5ex]current bounding box.center)}] 

\draw(-1.5,1.3) node[scale=1] {$f_{\{2\}}M =$};
\draw (0,0) node[scale=1] {$c$};
\draw (2,0) node[scale=1] {$c$};
\draw (1,1) node[scale=1] {$d$};
\draw (3,1) node[scale=1] {$d$};
\draw (0,2) node[scale=1] {$a$};
\draw (2,2) node[scale=1] {$a$};
\draw (1,3) node[scale=1] {$b$};
\draw (3,3) node[scale=1] {$b$};

\draw (0,0.3) -- (0,1.7); 
\draw (0.3,0) -- (1.7,0); 
\draw (2,0.3) -- (2,1.7); 
\draw (0.3,2) -- (1.7,2); 

\draw (1,1.3) -- (1,2.7); 
\draw (1.3,1) -- (2.7,1); 
\draw (3,1.3) -- (3,2.7); 
\draw (1.3,3) -- (2.7,3);

\draw (0.3,0.3) -- (0.7,0.7); 
\draw (0.3,2.3) -- (0.7,2.7); 
\draw (2.3,0.3) -- (2.7,0.7); 
\draw (2.3,2.3) -- (2.7,2.7); 
\end{tikzpicture}
\end{equation}
\end{center}
\end{example}

In this example, $(f_{\{2\}}(M))_{i_1, i_3} = M_{j_1, j_2}$ where $i_1=j_1$, $i_3 = j_2 $. Indeed, $(f_{\{2\}}(M))_{000} = (f_{\{2\}}(M))_{010} = M_{00} = a$.

\begin{remark} If $T=v_1\otimes\cdots\otimes v_d$ is a tensor of rank $1$, then its $J$-th forget tensor $f_J(T)$ is again a tensor of rank $1$. Indeed, in term of a basis $\{e_1,\dots,e_n\}$ of $V$, for which $v_1=a_1e_1+\dots+a_ne_n$,
\begin{equation}
f_k(T)= w_1\otimes\cdots\otimes w_{d'},
\end{equation} 
where $w_i=v_i$ if $i\notin J$ while $w_i=e_1+\cdots+e_n$ if $i\in J$.
\end{remark}

Notice that when $J=\emptyset$, the forget tensor $f_J(T)$ is equal to $T$.

\subsection{The Theorem}

Now we are ready to explain how one gets the total tensor of a network by multiplying with the BMP certain tensors  obtained from the activation tensors of the network.
\medskip

Assume we are given a network $\Nc(G, \mathcal{A})$ with $\left | V(G) \right | = q$ nodes and total ordering $(a=a_1, a_2,\dots, a_q=b)$. We call $a$ the \emph{source} and $b$ the \emph{sink} of the network.

For any node $a_i \in V(G)$, we define $P_i$ as the set of indexes of the parents of $a_i$ (see~\Cref{def:total}), so that $P_i\subset\{1,\dots,i-1\}$. Let $p_i$ be the cardinality of $P_i$ and $T_i$ be the activation tensor of the node $a_i$. According to~\Cref{def:activation}, tensor $T_i$ has order $p_i+1$.

\begin{definition}\label{proc} For any node $a_i \in V(G)$, we define the associated order-$q$ tensor $T_i$ with the following procedure.

{\it Step 1.} Start with the activation tensor $A_i$.

{\it Step 2.} If $i>1$, consider the $i$-dimensional forget tensor $A'_i=f_J(A_i)$, where $J=\{1,\dots,i-1\}\setminus P_i$. 

{\it Step 3.} If $i\neq q$, consider $A''_i=b(A'_i)$.

{\it Step 4.} Define $T_i$ as the order-$q$ forget tensor $T_i=f_H(A''_i)$, where $H=\{i+2,\dots,q\}.$ 
\end{definition}

The listed rules ensure that each activation tensor is expanded to order $q$ at the end of the procedure in~\Cref{proc}. A more detailed explanation follows:
\begin{itemize}
    \item the order of tensor $T_i$ depends on the number of parents $P_i$ of node $i$: the starting order of the tensor is $P_i + 1 \leq i$;
    \item $|J|$ has cardinality $i-1 - P_i$, therefore, applying $|J|$ forgets, we obtain a tensor of order $i$;
    \item if $i<q$, applying the blow we obtain a tensor of order $i+1$; 
    \item if $i+1<q$, we need to apply forgets from $i+2$ to $q$ to obtain a tensor of order $q$.
\end{itemize}

\begin{theorem}\label{main}

Let $\Nc(G, \mathcal{A})$ be a network with $\left | V(G) \right | = q$ nodes and total ordering $(a=a_1, a_2,\dots, a_q=b)$, such that $a$ is the source and $b$ is the sink. Moreover, assume that $T_1, \dots, T_q$ are the tensors obtained by applying procedure in~\Cref{proc} to the activation tensors $A_1, \dots, A_q \in \mathcal{A}$.
The BMP 
\begin{equation}
\circ (T_1, \dots, T_{q})
\end{equation}
is equal to the total tensor $N$ of the network.  
\end{theorem}

Before we provide the proof of the main theorem, let us illustrate the procedure with some examples.

\begin{example}\label{basec} Consider the basic case of a network in which there are only two nodes, $a$ the source and $b$ the sink.
The activation tensor of $a$ is just a vector $A_1=(\alpha_0,\alpha_1,\dots,\alpha_{n-1})$. The activation tensor of $b$ depends on the existence of an arrow from $a$ to $b$.

Assume the arrow $a\to b$ exists. Then the activation tensor $A_2 = (a_{ij})$ of $b$ is a matrix of type $n\times n$ which describes the reaction of $b$ to the input from $a$. 
If we apply the procedure of Definition~\ref{proc} to $A_1$, the final result is the tensor $A'_1=b (A_1)$, since $J=H=\emptyset$. By definition of blow, $A'_1$ is the diagonal matrix with the vector $A_1$ in the diagonal. Set $T_1 = A'_1$.
If we apply the procedure of Definition~\ref{proc} to $A_2$, since again $J=H=\emptyset$, we obtain the matrix $T_2 = A_2$ itself. Thus, the BMP $\circ (T_1, T_2)$, which corresponds to the usual matrix multiplication $T_1 \cdot T_2$, produces the matrix $N$ with $N_{ij}=(A_1)_i(A_2)_{i,j}$, which corresponds to the fact that $a$ casts the signal $i$ and $b$ reacts with $j$. In other words, $N$ is the tensor of the network. 

Assume the arrow $a\to b$ does not exist. Then the activation tensor $A_2$ of $b$ is a vector $A_2= (\alpha_0,\dots,\alpha_{n-1})$.
As above, if we apply the procedure of Definition~\ref{proc} to $A_1$, the final result is the tensor $A'_1=b (A_1)$, i.e. the diagonal matrix with the vector $A_1$ in the diagonal. Set $T_1 = A'_1$.
If we apply the procedure of Definition~\ref{proc} to $A_2$, we obtain the forget tensor $A'_2 =f_{\{1\}}(A_2)$, which is an $n\times n$ matrix whose rows are all equal to $A_2$, and set $T_2 = A'_2$. 
Thus, the BMP $\circ (T_1, T_2)$, which corresponds to the matrix multiplication $T_1 \cdot T_2$, produces the matrix $N$ with $N_{ij}=(A_1)_i(A_2)_{j}$, meaning that $a$ and $b$ are totally independent. Again, $N$ is clearly the tensor of the network. 

\end{example}

\begin{example} Consider the Markov network of Fig.\ref{fig:Markov_network}, and let us associate to $a$ the activation tensor $A_1=(\alpha,\beta)$, while the activation tensors $A_2, A_3$ of the nodes $b$ and $c$, respectively, are both equal to the Jukes- Cantor matrix $A = \begin{pmatrix}
    \alpha & \beta \\ \beta & \alpha
\end{pmatrix}$ of Example~\ref{qt1}. 
The total tensor of the network is thus given by the BMP of the tensors $T_1, T_2, T_3$ where $T_1=f_{\{3\}}(b(A_1))$,
$T_2 = b(A_2)$ and $T_3 = f_{\{1\}}(A_3)$. One computes 

\begin{center}
\begin{equation}
\begin{tikzpicture}[scale=0.75,baseline={([yshift=-.5ex]current bounding box.center)}] 
\draw (-1,1.3) node[scale=0.9] {$T_1=$};
\draw (0,0) node[scale=0.9] {$0$};
\draw (2,0) node[scale=0.9] {$\beta$};
\draw (1,1) node[scale=0.9] {$0$};
\draw (3,1) node[scale=0.9] {$\beta$};
\draw (0,2) node[scale=0.9] {$\alpha$};
\draw (2,2) node[scale=0.9] {$0$};
\draw (1,3) node[scale=0.9] {$\alpha$};
\draw (3,3) node[scale=0.9] {$0$};
\draw (0,0.3) -- (0,1.7); 
\draw (0.3,0) -- (1.7,0); 
\draw (2,0.3) -- (2,1.7); 
\draw (0.3,2) -- (1.7,2);
\draw (1,1.3) -- (1,2.7); 
\draw (1.3,1) -- (2.7,1); 
\draw (3,1.3) -- (3,2.7); 
\draw (1.3,3) -- (2.7,3);
\draw (0.3,0.3) -- (0.7,0.7); 
\draw (0.3,2.3) -- (0.7,2.7); 
\draw (2.3,0.3) -- (2.7,0.7); 
\draw (2.3,2.3) -- (2.7,2.7);

\draw (4,1.3) node[scale=0.9] {$,\ \ T_2=$};
\draw (5,0) node[scale=0.9] {$0$};
\draw (7,0) node[scale=0.9] {$0$};
\draw (6,1) node[scale=0.9] {$\beta$};
\draw (8,1) node[scale=0.9] {$\alpha$};
\draw (5,2) node[scale=0.9] {$\alpha$};
\draw (7,2) node[scale=0.9] {$\beta$};
\draw (6,3) node[scale=0.9] {$0$};
\draw (8,3) node[scale=0.9] {$0$};
\draw (5,0.3) -- (5,1.7); 
\draw (5.3,0) -- (6.7,0); 
\draw (7,0.3) -- (7,1.7); 
\draw (5.3,2) -- (6.7,2);
\draw (6,1.3) -- (6,2.7); 
\draw (6.3,1) -- (7.7,1); 
\draw (8,1.3) -- (8,2.7); 
\draw (6.3,3) -- (7.7,3);
\draw (5.3,0.3) -- (5.7,0.7); 
\draw (5.3,2.3) -- (5.7,2.7); 
\draw (7.3,0.3) -- (7.7,0.7); 
\draw (7.3,2.3) -- (7.7,2.7); 

\draw (9,1.3) node[scale=0.9] {$,\ \ T_3=$};
\draw (10,0) node[scale=0.9] {$\alpha$};
\draw (12,0) node[scale=0.9] {$\beta$};
\draw (11,1) node[scale=0.9] {$\beta$};
\draw (13,1) node[scale=0.9] {$\alpha$};
\draw (10,2) node[scale=0.9] {$\alpha$};
\draw (12,2) node[scale=0.9] {$\beta$};
\draw (11,3) node[scale=0.9] {$\beta$};
\draw (13,3) node[scale=0.9] {$\alpha$};
\draw (10,0.3) -- (10,1.7); 
\draw (10.3,0) -- (11.7,0); 
\draw (12,0.3) -- (12,1.7); 
\draw (10.3,2) -- (11.7,2); 
\draw (11,1.3) -- (11,2.7); 
\draw (11.3,1) -- (12.7,1); 
\draw (13,1.3) -- (13,2.7); 
\draw (11.3,3) -- (12.7,3);
\draw (10.3,0.3) -- (10.7,0.7); 
\draw (10.3,2.3) -- (10.7,2.7); 
\draw (12.3,0.3) -- (12.7,0.7); 
\draw (12.3,2.3) -- (12.7,2.7); 

\end{tikzpicture},
\end{equation}
\end{center}
and the product $\circ(T_1, T_2, T_3)$ gives
\begin{center}
\begin{equation}
\begin{tikzpicture}[scale=0.75,baseline={([yshift=-.5ex]current bounding box.center)}] 
\draw (-1,1.3) node[scale=0.9] {$N= $};
\draw (0,0) node[scale=0.9] {$\alpha\beta^2$};
\draw (2,0) node[scale=0.9] {$\alpha\beta^2$};
\draw (1,1) node[scale=0.9] {$\beta^3$};
\draw (3,1) node[scale=0.9] {$\alpha^2\beta$};
\draw (0,2) node[scale=0.9] {$\alpha^3$};
\draw (2,2) node[scale=0.9] {$\alpha\beta^2$};
\draw (1,3) node[scale=0.9] {$\alpha^2\beta$};
\draw (3,3) node[scale=0.9] {$\alpha^2\beta$};
\draw (0,0.3) -- (0,1.7); 
\draw (0.5,0) -- (1.5,0); 
\draw (2,0.3) -- (2,1.7); 
\draw (0.5,2) -- (1.5,2); 
\draw (1,1.3) -- (1,2.7); 
\draw (1.5,1) -- (2.5,1); 
\draw (3,1.3) -- (3,2.7); 
\draw (1.5,3) -- (2.5,3);
\draw (0.3,0.3) -- (0.7,0.7); 
\draw (0.3,2.3) -- (0.7,2.7); 
\draw (2.3,0.3) -- (2.7,0.7); 
\draw (2.3,2.3) -- (2.7,2.7); 
\end{tikzpicture},
\end{equation}
\end{center}
which represents a twisted cubic curve in the projective space of tensors of type $2\times 2\times 2$, sitting in the linear space of equations $N_{010}=N_{100}=N_{110}$ and $N_{001}=N_{011}=N_{111}$.

Note that both the left and the rigth faces of the tensor have rank $1$, corresponding to the fact that $c$ is independent of $a$ given $b$.
\smallskip

If we remove from the network the arrow that joins $b$ and $a$ we obtain a new network as in~\Cref{fig:network}.
\begin{figure}[!h]
\centering
\begin{tikzpicture}[
> = stealth, 
shorten > = 1pt, 
auto,
node distance = 3cm, 
semithick 
]
\tikzstyle{every state}=[
draw = black,
thick,
fill = white,
minimum size = 12mm
]
\node[state] (1) {$a$};
\node[state] (2) [right of=1] {$b$};
\node[state] (3) [right of=2] {$c$}; 

\path[->] (1) edge node {} (2);
\end{tikzpicture}
\caption{\textbf{Network with three nodes, two of which are independent.} Tensor network with three nodes $b, c, a$ and their associated activation tensors $A_1, A_2, A_3$. Node $a$ transmits a binary signal to node $b$. Nodes $a$ and $c$ are independent.}
\label{fig:network}
\end{figure}
\bigskip

Before assigning activation tensors, a total ordering of the nodes, consistent with the partial ordering, must be established. We assume the sorting $(a, b, c)$.
At this point, taking the activation tensor $A_1, A_2$ as above and $A_3=A_1$, we get that the total tensor of the network is  given by the BMP $\circ(T_1, T_2, T_3)$, where $T_1 = A'_1$ and $T_2 = A'_2$ are as above and $ T_3 = A'=f_{\{1,2\}}(A_3)$, i.e.
\begin{center}
\begin{equation}
\begin{tikzpicture}[scale=0.75,baseline={([yshift=-.5ex]current bounding box.center)}] 
\draw (-1,1.3) node[scale=0.9] {$T_3=$};
\draw (0,0) node[scale=0.9] {$\alpha$};
\draw (2,0) node[scale=0.9] {$\alpha$};
\draw (1,1) node[scale=0.9] {$\beta$};
\draw (3,1) node[scale=0.9] {$\beta$};
\draw (0,2) node[scale=0.9] {$\alpha$};
\draw (2,2) node[scale=0.9] {$\alpha$};
\draw (1,3) node[scale=0.9] {$\beta$};
\draw (3,3) node[scale=0.9] {$\beta$};
\draw (0,0.3) -- (0,1.7); 
\draw (0.3,0) -- (1.7,0); 
\draw (2,0.3) -- (2,1.7); 
\draw (0.3,2) -- (1.7,2); 
\draw (1,1.3) -- (1,2.7); 
\draw (1.3,1) -- (2.7,1); 
\draw (3,1.3) -- (3,2.7); 
\draw (1.3,3) -- (2.7,3);
\draw (0.3,0.3) -- (0.7,0.7); 
\draw (0.3,2.3) -- (0.7,2.7); 
\draw (2.3,0.3) -- (2.7,0.7); 
\draw (2.3,2.3) -- (2.7,2.7); 
\end{tikzpicture}.
\end{equation}
\end{center}

The product $\circ(T_1, T_2, T_3)$ corresponds to
\begin{center}
\begin{equation}
\begin{tikzpicture}[scale=0.75,baseline={([yshift=-.5ex]current bounding box.center)}] 
\draw (-1,1.3) node[scale=0.9] {$N=$};
\draw (0,0) node[scale=0.9] {$\alpha\beta^2$};
\draw (2,0) node[scale=0.9] {$\alpha^2\beta$};
\draw (1,1) node[scale=0.9] {$\beta^3$};
\draw (3,1) node[scale=0.9] {$\alpha\beta^2$};
\draw (0,2) node[scale=0.9] {$\alpha^3$};
\draw (2,2) node[scale=0.9] {$\alpha^2\beta$};
\draw (1,3) node[scale=0.9] {$\alpha^2\beta$};
\draw (3,3) node[scale=0.9] {$\alpha\beta^2$};
\draw (0,0.3) -- (0,1.7); 
\draw (0.5,0) -- (1.5,0); 
\draw (2,0.3) -- (2,1.7); 
\draw (0.3,2) -- (1.5,2); 
\draw (1,1.3) -- (1,2.7); 
\draw (1.3,1) -- (2.5,1); 
\draw (3,1.3) -- (3,2.7); 
\draw (1.5,3) -- (2.5,3);
\draw (0.3,0.3) -- (0.7,0.7); 
\draw (0.3,2.3) -- (0.7,2.7); 
\draw (2.3,0.3) -- (2.7,0.7); 
\draw (2.3,2.3) -- (2.7,2.7); 
\end{tikzpicture},
\end{equation}
\end{center}
which represents a different twisted cubic curve.

Note that not only both the left and right faces of the tensor have rank $1$, \emph{i.e.} $a$ is independent of $b$ given $c$, but also the upper and lower faces have rank $1$, since $a,c$ are now independent.
\end{example}

\begin{example} \label{ex:27}
It is obvious that in a network without arrows all the activation tensors of the nodes are vectors, and the tensor product of these vectors defines a rank one tensor $N$ which is the total tensor of the network.

If a $n$-ary network with no arrows has $q$ nodes $a_1, \dots, a_{q}$, with total ordering $(a_1, \dots, a_q)$, and $w_i\in \C^n$ is the activation vector of the node $a_i$, $i = 1, \dots, q$, then Theorem~\ref{main} shows that the BMP of the order-$n$ tensors $T_i$'s, obtained by the $w_i$'s with the application of the procedure in~\Cref{proc}, must be equal to $w_1\otimes\cdots\otimes w_q$.
This can be also proved directly.
Namely, we have:
$$T_1=f_{\{3,\dots,q\}}(b(w_1)),\quad T_j=f_{\{j+2,\dots,q\}}(b(f_{\{1,\dots j-1\}}(w_j))) \ \text{for} \ j = 1, \dots, q-2;$$  
$$T_{q-1} = b(f_{\{1,\dots,q-2\}}(w_{q-1})),\quad 
T_q= f_{\{1,\dots, q-1\}}(w_q).$$

Thus, assuming $w_i=(w_{i1},\dots,w_{in})$ for all $i = 1, \dots, q$, the computation of the product $N = \circ (T_1 \cdots T_{q})$ goes as follows:
$$N_{i_1,\dots,i_q} = \sum_{h=0}^{n-1}(T_1)_{i_1,h,i_3,\dots,q}\cdots (T_{q-1})_{i_1,\dots,i_{q-1},h}(T_q)_{h,i_2,\dots,i_q}.$$
Since the last operation used to obtain $T_{q-1}$ is a blow, the only summand that survives is the one with $h=i_1$. \emph{I.e.}
$$N_{i_1,\dots,i_q} = (T_1)_{i_1,i_1,i_3,\dots,i_q}\cdots (T_{q-1})_{i_1,\dots,i_{q-1},i_1} (T_q)_{i_1,i_2,\dots,i_q}.$$
From our construction, one gets $$(T_q)_{i_1,i_2,\dots,i_q}=(f_{\{1,\dots,q-1\}}(w_q))_{i_1,i_2,\dots,i_q} = w_{q i_q}.$$ 
Analogously, for $j=1,\dots,q-2$ ,
$$(T_j)_{i_1,i_2,\dots,i_q}=(f_{\{j+2,\dots,q\}}(b(f_{\{1,\dots, j-1\}}(w_j))))_{i_1,i_2,\dots,i_{j-1}, i_1,i_{j+1},\dots,i_q},$$ 
where the index $i_1$ in the middle appears in position $j$. Thus, $$(f_{\{j+2,\dots,q\}}(b(f_{\{1,\dots, j-1\}}(w_j))))_{i_1,i_2,\dots, i_{j-1},i_1,i_{j+1}\dots,i_q}= 
(b(f_{\{1,\dots, j-1\}}(w_j)))_{i_1,i_2,\dots,i_{j-1}, i_1},$$ and $$(f_{\{1,\dots, j-1\}}(w_j))_{i_1,i_2,\dots,i_{j-1}}=w_{j i_j}.$$ With the same argument, one proves that $$(T_1)_{i_1,i_1,i_3,\dots,i_q}=w_{1 i_1}, \ \text{and }(T_{q-1})_{i_1,i_2,\dots,i_{q-1},i_1}=w_{q-1\ i_{q-1}}.$$ Then 
$$ N_{i_1,\dots,i_q} = w_{1 i_1}\cdots w_{q i_q}.$$
\end{example}
\medskip

{\it Proof of Theorem~\ref{main}.} \ \
We work by induction on the number of nodes, the case $q=1$ being trivial. We could also use Example~\ref{basec} as a basis for the induction, starting with $q=2$.

Assume that the theorem holds for any network with at most $q-1$ nodes, and pick a network $\Nc(G, \mathcal{A})$ with $\left | V(G) \right | = q$, with total ordering $(a_1, \dots, a_q)$.

If we exclude the sink $a_q=b$, then we obtain a network $\Nc'(G', \mathcal{A}')$ for which the theorem holds, by induction. Use on the nodes of $G'$ the ordering defined for the nodes of $G$.

Let $N$ be the total tensor of $\Nc$ and $N'$ be the total tensor of $\Nc'$. Then, the relation between $N$ and $N'$ can be described as follows. Call $P=P_q$ the set of indexes corresponding to $parents(b)$, ordered in the ordering defined for $G$. If $B$ is the activation tensor of $b$, then 
$$ N_{i_1,\dots,i_q}= N'_{i_1,\dots,i_{q-1}}(B)_{P(i_1,\dots,i_q)}$$
where $P(i_1,\dots,i_q)$ means that we pick in the sequence $(i_1,\dots,i_q)$ only the indexes corresponding to values in $P$.

Let $T'_i$ be the tensor associated with $a_i$ constructed as in the procedure of~\Cref{proc} for the network $\Nc'$, and let $T_i$ be the tensor associated with $a_i$ obtained by the same procedure for the network $\Nc$, for $i = 1, \dots q-1$. Then $T'_i$ and $T_i$ are linked by the following relations:
$$ \begin{cases}  T_i &= f_{\{q\}}(T'_i) \quad \text{ if } i=1,\dots,q-2;\\
T_{q-1} &= b(T'_{q-1}).
\end{cases}
$$
Moreover, by applying the procedure in~\Cref{proc} to the activation tensor $B$, we obtain the tensor $T_q$ associated to node $b$:
\begin{equation}\label{eq:forget_B}
     T_q = f_J(B) \quad\text{ where } J=\{1\dots,q-1\}\setminus P.
\end{equation}
Now, we only need to perform the required calculations.

The BMP $M=\circ  ( T_1,\cdots, T_{q-1}, T_q)$ has the following structure:
$$M_{i_1,\dots,i_q}=\sum_{h=0}^{n-1} (T_1)_{i_1,h,i_3,\dots,i_q}\cdots (T_{q-1})_{i_1,\dots,i_{q-1},h}(T_q)_{h,i_2,\dots,i_q}.$$
Since $T_{q-1}$ is the blow of a tensor, the unique summand which survives in the previous expression is the one with $h=i_1$.
Thus
$$M_{i_1,\dots,i_q}= (T_1)_{i_1,i_1,i_3,\dots,i_q}\cdots (T_{q-1})_{i_1,\dots,i_{q-1},i_1}(T_q)_{i_1,i_2,\dots,i_q}.$$
By construction, $(T_q)_{i_1,i_2,\dots,i_q} = (f_J(B))_{i_1,i_2,\dots,i_q} = (B)_{P(i_1,\dots,i_q)}.$ Moreover, one gets  $(T_j)_{i_1,\dots,i_{j-1},i_1,i_{j+1},\dots,i_q}=(T'_j)_{i_1,\dots,i_{j-1},i_1,i_{j+1},\dots,i_{q-1}}$ for $j=1,\dots, q-2$, and $(T_{q-1})_{i_1,\dots,i_{q-1},i_1}=(T'_{q-1})_{i_1,\dots,i_{q-1}}$ for the last factor in the product.

By induction, we know that $N'$ is given by the BMP $ \circ (T'_1\cdots T'_{q-2} T'_{q-1})$, thus 
$$N'_{i_1,\dots,i_q}=\sum_{h=0}^{n-1} (T'_1)_{i_1,h,i_3,\dots,i_{q-1}}\cdots (T_{q-2})_{i_1,\dots,i_{q-2},h}(T'_{q-1})_{h,i_2,\dots,i_{q-1}}.$$
As already introduced in~\Cref{ex:27}, the unique summand that survives has $h=i_1$, so 
$$N'_{i_1,\dots,i_q}=(T'_{q-1})_{i_1,\dots,i_{q-1}}(T'_1)_{i_1,i_1,i_3,\dots,i_{d-1}}\cdots (T'_{q-2})_{i_1,\dots,i_{q-2},i_1}.$$
Comparing with the previous expressions, it follows that
$$N'_{i_1,\dots,i_q}=(T_1)_{i_1,i_1,i_3,\dots,i_q}\cdots (T_{q-2})_{i_1,\dots,i_{q-1},i_1}(T_{q-1})_{i_1,\dots,i_{q-1},i_1}.$$
Finally,
\begin{multline*}
N_{i_1,\dots,i_q}= N'_{i_1,\dots,i_{d-1}}(B)_{P(i_1,\dots,i_q)}=\\
= (T_1)_{i_1,i_1,i_3,\dots,i_q}\cdots (T_{q-2})_{i_1,\dots,i_{q-1},i_1} (T_{q-1})_{i_1,\dots,i_{q-1},i_1}(T_q)_{i_1,i_2,\dots,i_q}
= M_{i_1,\dots,i_q}.
\end{multline*}
The claim follows. \hfill \qed

\begin{example}  Consider the binary network in~\Cref{fig}
\begin{figure}[ht]
\centering
\begin{tikzpicture}[scale=0.75,baseline={([yshift=-.5ex]current bounding box.center)}][
> = stealth, 
shorten > = 1pt, 
auto,
node distance = 3cm, 
semithick 
]
\tikzstyle{every state}=[
draw = black,
thick,
fill = white,
minimum size = 3mm
]
\node[state] (1) {$a$};
\node[state] (2) [below of=1] {$b$};
\node[state] (3) [right of=2] {$c$}; 
\node[state] (4) [below of=3] {$d$}; 
\node[state] (5) [right of=4] {$e$}; 

\path[->] (1) edge node {} (2);
\path[->] (2) edge node {} (3); 
\path[->] (2) edge node {} (5); 
\path[->] (1) edge node {} (3); 
\path[->] (3) edge node {} (4); 
\path[->] (4) edge node {} (5); 
\end{tikzpicture}
\caption{\textbf{Network with five nodes.} A network with nodes $a$, $b$, $c$, $d$, $e$. The source is $a$ and the sink is $e$.}
\label{fig}
\end{figure}
with the (unique) ordering $(a, b, c, d, e)$. Define the activation tensors, which depend on two parameters, $B$ at the node $b$, $A$ at the node $a$, etc., as follows:
$$A=(\alpha,\beta),\qquad  B=D=\begin{pmatrix} \alpha & \beta \\ \beta & \alpha \end{pmatrix}, $$
\begin{center}
\begin{tikzpicture}[scale=0.75,baseline={([yshift=-.5ex]current bounding box.center)}] 
\draw (-1.6,1.4) node[scale=1] {$C = E =$};
\draw (0,0) node[scale=1] {$\beta$};
\draw (2,0) node[scale=1] {$\beta$};
\draw (1,1) node[scale=1] {$\alpha$};
\draw (3,1) node[scale=1] {$\alpha$};
\draw (0,2) node[scale=1] {$\alpha$};
\draw (2,2) node[scale=1] {$\beta$};
\draw (1,3) node[scale=1] {$\beta$};
\draw (3,3) node[scale=1] {$\alpha$};
\draw (0,0.3) -- (0,1.7); 
\draw (0.3,0) -- (1.7,0); 
\draw (2,0.3) -- (2,1.7); 
\draw (0.3,2) -- (1.7,2); 

\draw (1,1.3) -- (1,2.7); 
\draw (1.3,1) -- (2.7,1); 
\draw (3,1.3) -- (3,2.7); 
\draw (1.3,3) -- (2.7,3);

\draw (0.3,0.3) -- (0.7,0.7); 
\draw (0.3,2.3) -- (0.7,2.7); 
\draw (2.3,0.3) -- (2.7,0.7); 
\draw (2.3,2.3) -- (2.7,2.7); 
\end{tikzpicture}.
\end{center}
Then, the total tensor $N$ is given by the BMP
$$ N = \circ(A', B', C', D', E'),$$
where $A', B', C', D', E'$ are the tensors of order $5$ determined as in the procedure of~\Cref{proc}. One computes:
$$ A' = f_{\{3,4,5\}}(b(A))\quad B'=f_{\{4,5\}}(b(B)), \quad C'=f_{\{5\}}(b(C)),\quad  D' = b(f_{\{1,2\}}(D)),\quad E'=f_{\{1,3\}}(E).$$
It follows that $A',B',C',D',E'$ are given by

$$\small  \begin{matrix}                            A'_{00000} = \alpha &   A'_{00001} = \alpha & A'_{00010} =\alpha & A'_{00011} =\alpha & A'_{00100} = \alpha &   A'_{00101} =\alpha  & A'_{00110} =\alpha & A'_{00111} =\alpha \\
A'_{01000} = 0 &   A'_{01001} = 0  & A'_{01010} = 0 & A'_{01011} = 0 & A'_{01100} = 0 &   A'_{01101} =  0 & A'_{01110} = 0 & A'_{01111} =0 \\
A'_{10000} = 0 &   A'_{10001} = 0  & A'_{10010} = 0 & A'_{10011} = 0 & A'_{10100} = 0 &   A'_{10101} =  0 & A'_{10110} = 0 & A'_{10111} = 0 \\
A'_{11000} = \beta &   A'_{11001} = \beta & A'_{11010} = \beta & A'_{11011} = \beta & A'_{11100} = \beta &   A'_{11101} =  \beta & A'_{11110} = \beta & A'_{11111} = \beta \end{matrix} $$
$$\small  \begin{matrix}                           B'_{00000} =\alpha &   B'_{00001} = \alpha   & B'_{00010} = \alpha & B'_{00011} = \alpha & B'_{00100} = 0 &   B'_{00101} = 0 & B'_{00110} =0& B'_{00111} = 0 \\
B'_{01000} = \beta &   B'_{01001} =\beta  & B'_{01010} =\beta & B'_{01011} =\beta & B'_{01100} = 0 &   B'_{01101} = 0 & B'_{01110} = 0& B'_{01111} =0 \\
B'_{10000} = 0&   B'_{10001} = 0 & B'_{10010} = 0& B'_{10011} = 0& B'_{10100} = \beta &   B'_{10101} = \beta & B'_{10110} =\beta & B'_{10111} = \beta \\
B'_{11000} = 0&   B'_{11001} = 0 & B'_{11010} = 0& B'_{11011} = 0& B'_{11100} = \alpha &   B'_{11101} =\alpha   & B'_{11110} =\alpha & B'_{11111} =\alpha  \end{matrix} $$
$$\small  \begin{matrix}                            C'_{00000} = \alpha&   C'_{00001} = \alpha & C'_{00010} = 0& C'_{00011} = 0 & C'_{00100} = \beta &   C'_{00101} = \beta & C'_{00110} = 0 & C'_{00111} = 0 \\
C'_{01000} = \beta&   C'_{01001} = \beta & C'_{01010} = 0 & C'_{01011} = 0& C'_{01100} = \alpha&   C'_{01101} =  \alpha& C'_{01110} =0 & C'_{01111} =0 \\
C'_{10000} = 0&   C'_{10001} = 0 & C'_{10010} = \beta & C'_{10011} = \beta & C'_{10100} = 0 &   C'_{10101} =0  & C'_{10110} =\alpha & C'_{10111} =\alpha \\
C'_{11000} =0 &   C'_{11001} = 0 & C'_{11010} =\beta  & C'_{11011} = \beta & C'_{11100} = 0&   C'_{11101} = 0 & C'_{11110} = \alpha & C'_{11111} = \alpha \end{matrix} $$
$$\small  \begin{matrix}                            D'_{00000} = \alpha &   D'_{00001} = 0  & D'_{00010} =\beta & D'_{00011} = 0 & D'_{00100} = \beta &   D'_{00101} = 0 & D'_{00110} = \alpha & D'_{00111} = 0 \\
D'_{01000} = \alpha &   D'_{01001} = 0  & D'_{01010} =\beta & D'_{01011} = 0 & D'_{01100} = \beta &   D'_{01101} = 0 & D'_{01110} = \alpha & D'_{01111} = 0 \\
D'_{10000} = 0 &   D'_{10001} = \alpha & D'_{10010} = 0 & D'_{10011} = \beta & D'_{10100} = 0 &   D'_{10101} = \beta & D'_{10110} = 0 & D'_{10111} = \alpha \\
D'_{11000} = 0 &   D'_{11001} = \alpha & D'_{11010} = 0 & D'_{11011} = \beta & D'_{11100} = 0 &   D'_{11101} = \beta & D'_{11110} = 0 & D'_{11111} = \alpha \end{matrix} $$
$$\small  \begin{matrix}                           E'_{00000} =\alpha &   E'_{00001} =\beta  & E'_{00010} =\beta & E'_{00011} = \alpha& E'_{00100} =\alpha &   E'_{00101} =\beta  & E'_{00110} =\beta & E'_{00111} =\alpha \\
E'_{01000} =\beta &   E'_{01001} =\alpha  & E'_{01010} =\beta & E'_{01011} = \alpha& E'_{01100} =\beta &   E'_{01101} =\alpha  & E'_{01110} =\beta & E'_{01111} =\alpha\\
E'_{10000} =\alpha &   E'_{10001} =\beta  & E'_{10010} =\beta & E'_{10011} = \alpha& E'_{10100} =\alpha &   E'_{10101} = \beta & E'_{10110} =\beta & E0'_{10111} =\alpha \\
E'_{11000} =\beta &   E'_{11001} =\alpha & E'_{11010} =\beta & E'_{11011} = \alpha& E'_{11100} =\beta &   E'_{11101} =\alpha  & E'_{11110} =\beta & E'_{11111} =\alpha \end{matrix} $$
\smallskip
and, by computing the BMP $\circ(A', B', C', D', E')$, we obtain the correct evaluation of $N$, which results to be
$$\small \begin{matrix}                            N_{00000} = \alpha^5 &   N_{00001} = \alpha^4\beta & N_{00010} = \alpha^3\beta^2 & N_{00011} = \alpha^4\beta \\ N_{00100} = \alpha^3\beta^2 &   N_{00101} = \alpha^2\beta^3  & N_{00110} = \alpha^3\beta^2 & N_{00111} = \alpha^4\beta \\
N_{01000} = \alpha^2\beta^3 &   N_{01001} = \alpha^3\beta^2  & N_{01010} = \alpha\beta^4 & N_{01011} = \alpha^2\beta^3 \\ N_{01100} = \alpha^2\beta^3 &   N_{01101} = \alpha^3\beta^2  & N_{01110} = \alpha^3\beta^2 & N_{01111} =\alpha^4\beta \\
N_{10000} = \alpha^2\beta^3 &   N_{10001} = \alpha\beta^4  & N_{10010} = \beta^5 & N_{10011} = \alpha\beta^4 \\ N_{10100} = \alpha^2\beta^3 &   N_{10101} = \alpha\beta^4  & N_{10110} = \alpha^2\beta^3 & N_{10111} =\alpha^3\beta^2 \\
N_{11000} = \alpha^2\beta^3 &   N_{11001} = \alpha^3\beta^2  & N_{11010} = \alpha\beta^4 & N_{11011} = \alpha^2\beta^3 \\ N_{11100} = \alpha^2\beta^3 &   N_{11101} = \alpha^3\beta^2  & N_{11110} = \alpha^3\beta^2 & N_{11111} = \alpha^4\beta.
\end{matrix}$$
Observe that, in the projective space of tensors of type $2\times 2\times 2\times 2\times 2$, the tensors defined by $N$, when $\alpha,\beta$ vary, describe a rational normal curve of degree $5$ (higly degenerate).
\end{example}

\subsection{Software}

A Python software has been implemented for computing the BMP of tensors of generic order $d$, as in~\Cref{BMP}, and for performing the procedure in~\Cref{proc}.

Given a network, users can provide node activation tensors as input and then perform blow and forget operations to calculate the tensors associated with the nodes. The BMP of these tensors, according to~\Cref{main}, outputs the final tensor of the network. From a computational point of view, there are no limits to be imposed on the number of nodes in the networks, assuming that their underlying graph is directed and acyclic. Tests on networks consisting of 6 or 7 nodes show minimal execution times. The software is open-source and freely available at: \url{https://github.com/MarzialiS/BMP-Network}.

\section*{Acknowledgements}
The authors want to warmly thank Luke Oeding, who introduced them to the early works of Batthacharya and Mesner, Fulvio Gesmundo, and Edinah Gnang, for useful discussions on the subject of the paper.
G.A.D. acknowledges the support provided by the European Union - NextGenerationEU, in the framework of the iNEST - Interconnected Nord-Est Innovation Ecosystem (iNEST ECS00000043 – CUP G93C22000610007) project and its CC5 Young Researchers initiative. S.S. and G.R. acknowledge the support provided by PRIN "FaReX - Full and Reduced order modelling of coupled systems: focus on non-matching methods and automatic learning" project, and by INdAM-GNCS 2019–2020 projects and PON "Research and Innovation on Green related issues" FSE REACT-EU 2021 project. The views and opinions expressed are solely those of the author and do not necessarily reflect those of the European Union, nor can the European Union be held responsible for him.
In addition, G.A.D. would like to acknowledge INdAM–GNCS.

\bibliographystyle{elsarticle-num} 
\bibliography{references}

\begin{bibdiv}
\begin{biblist}

\bib{B}{article}{
      author={Bhattacharya, Prabir},
       title={A new 3-{D} transform using a ternary product},
        date={1995},
     journal={IEEE transactions on signal processing},
      volume={43},
      number={12},
       pages={3081\ndash 3084},
}

\bib{BC}{book}{
      author={Bocci, Cristiano},
      author={Chiantini, Luca},
      author={others},
       title={An introduction to algebraic statistics with tensors},
   publisher={Springer},
        date={2019},
      volume={1},
}

\bib{DI}{book}{
      author={Diestel, Reinhard},
       title={Graph theory},
   publisher={Springer (print edition); Reinhard Diestel (eBooks)},
        date={2024},
}

\bib{G}{article}{
      author={Gnang, Edinah~K},
       title={A combinatorial approach to hypermatrix algebra},
        date={2014},
     journal={arXiv preprint arXiv:1403.3134},
}

\bib{GF1}{article}{
      author={Gnang, Edinah~K},
      author={Filmus, Yuval},
       title={On the spectra of hypermatrix direct sum and {K}ronecker products
  constructions},
        date={2017},
     journal={Linear Algebra and its Applications},
      volume={519},
       pages={238\ndash 277},
}

\bib{GF2}{article}{
      author={Gnang, Edinah~K},
      author={Filmus, Yuval},
       title={On the {B}hattacharya-{M}esner rank of third order
  hypermatrices},
        date={2020},
     journal={Linear Algebra and its Applications},
      volume={588},
       pages={391\ndash 418},
}

\bib{H}{book}{
      author={Hogben, Leslie},
       title={Handbook of linear algebra},
   publisher={CRC press},
        date={2006},
}

\bib{landsberg2011tensors}{book}{
      author={Landsberg, Joseph~M},
       title={Tensors: geometry and applications: geometry and applications},
   publisher={American Mathematical Soc.},
        date={2011},
      volume={128},
}

\bib{M}{article}{
      author={Marziali, Sara},
       title={Boolean operators and neural networks},
        date={2024/10/01},
     journal={ANNALI DELL'UNIVERSITA' DI FERRARA},
      volume={70},
      number={4},
       pages={1767\ndash 1783},
         url={https://doi.org/10.1007/s11565-024-00541-5},
}

\bib{BM1}{article}{
      author={Mesner, Dale~M},
      author={Bhattacharya, Prabir},
       title={Association schemes on triples and a ternary algebra},
        date={1990},
     journal={Journal of Combinatorial Theory, Series A},
      volume={55},
      number={2},
       pages={204\ndash 234},
}

\bib{BM2}{article}{
      author={Mesner, Dale~M},
      author={Bhattacharya, Prabir},
       title={A ternary algebra arising from association schemes on triples},
        date={1994},
     journal={Journal of algebra},
      volume={164},
      number={3},
       pages={595\ndash 613},
}

\bib{TKMP}{article}{
      author={Tian, Fan},
      author={Kilmer, Misha~E},
      author={Miller, Eric},
      author={Patra, Abani},
       title={{Tensor BM-Decomposition for Compression and Analysis of
  Spatio-Temporal Third-order Data}},
        date={2023},
     journal={arXiv preprint arXiv:2306.09201},
}

\end{biblist}
\end{bibdiv}



\end{document}